\documentclass{article}
\usepackage[margin=1in]{geometry}
\usepackage{amsmath, amsfonts, amssymb, amsthm}
\usepackage{hyperref}
\usepackage{color}

\newtheorem{thm}{Theorem}[section]
\newtheorem{prop}[thm]{Proposition}

\newtheorem{cor}[thm]{Corollary}
\newtheorem{fact}[thm]{(almost) Theorem}

\newtheorem{lemma}[thm]{Lemma}

\newtheorem{preremark}[thm]{Remark}
\newenvironment{remark}{\begin{preremark}\rm}{\medskip \end{preremark}}

\numberwithin{equation}{section}

\newcommand{\R}{\mathbb R}
\newcommand{\T}{\mathcal S}

\DeclareMathOperator{\tr}{tr}

\title{Singular solutions to parabolic equations in nondivergence form}
\author{Luis Silvestre}

\begin{document}

\maketitle

\begin{abstract}
For any $\alpha \in (0,1)$, we construct an example of a solution to a parabolic equation with \emph{measurable} coefficients in \textbf{two} space dimensions which has an isolated singularity and is not better that $C^\alpha$. We prove that there exists no solution to a fully nonlinear uniformly parabolic equation, in any dimension, which has an isolated singularity where it is not $C^2$ while it is analytic elsewhere, and it is homogeneous in $x$ at the time of the singularity. We build an example of a non homogeneous solution to a fully nonlinear uniformly parabolic equation with an isolated singularity, which we verify with the aid of a numerical computation.
\end{abstract}

\section{Introduction}
In this work, we are interested in the emergence of singularities from the flow of parabolic equations. We study two related types of equations. The equations with \emph{measurable} coefficients have the form
\begin{equation} \label{e:measurable_coefficients}
u_t - a_{ij}(t,x) \partial_{ij} u = 0.
\end{equation}
Here the coefficients $a_{ij}$ satisfy the ellipticity condition $\lambda \mathrm{I} \leq \{a_{ij}(t,x)\} \leq \Lambda \mathrm{I}$, for every point $(t,x)$ in the domain of the equation. No regularity is assumed for $a_{ij}$ with respect to either $x$ or $t$.

The other class of equations is that of translation-invariant fully nonlinear parabolic equations of the form
\begin{equation} \label{e:fully_nonlinear}
u_t - F(D^2 u) = 0.
\end{equation}
Here, we always assume that the function $F$ is uniformly elliptic in the sense that for any pair of symmetric matrices $A,B \in \R^{d \times d}$, if $B \geq 0$, we have $\lambda \tr B \leq F(A+B) - F(A) \leq \Lambda \tr B$.

Centered around the work of Krylov and Safonov \cite{ks}, there is a well developed regularity theory for parabolic equations in nondivergence form. Solutions to an equation with measurable coefficients like \eqref{e:measurable_coefficients} are known to be H\"older continuous. Solutions to a fully nonlinear parabolic equation like \eqref{e:fully_nonlinear} are known to be $C^{1+\alpha}$, for some $\alpha>0$ depending on dimension and the ellipticity parameters $\lambda$ and $\Lambda$. Without further assumptions, there is no regularity estimate that ensures $D^2u$ to be well defined.

There are some important examples that show that our currently known regularity results are sharp for elliptic equations in non-divergence form. These examples can also be interpreted as singular solutions to fully nonlinear parabolic equations that are constant in time. However, what we seek in this work is to understand whether a solution to a parabolic equation may start smooth and flow into a singularity after some finite positive time. We are interested in constructing a solution to a parabolic equation which has an isolated singularity in space-time.

For elliptic equations in three dimensions, M. Safonov constructs an example in \cite{safonov} showing that the H\"older continuity result in his joint work with Krylov \cite{ks} is optimal. Precisely, for any $\alpha \in (0,1)$, he constructs a function $u : \R^3 \to \R$, homogeneous of degree $\alpha$, smooth in $\R^3 \setminus \{0\}$, so that
\begin{equation} \label{e:elliptic_measurable}
a_{ij}(x) \partial_{ij} u(x) = 0 \text{ in } \R^3 \setminus \{0\},
\end{equation}
for some uniforly elliptic coefficients $a_{ij}$. Moreover, he shows that this function $u$ can be approximated with smooth functions satisfying a uniformly elliptic equation in the full space. The singular function $u$ is also a viscosity solution to the corresponding inequalities for the Pucci operators (see Section \ref{s:preliminaries} below) in the full space $\R^3$.

For fully nonlinear elliptic equations, there is a series of examples of homogeneous solutions with an isolated singularity at the origin (see \cite{NV1,NV2,NV3,NV4,NV5,NV6,NVT}). In particular, in \cite{NVT}, Nadirashvili, Tkachev and Vl\u{a}du\c{t} construct a solution to a fully nonlinear elliptic equation with an isolated singularity in dimension five. All these examples are homogeneous functions. 

Let us state our first main result concerning equations with \emph{measurable} coefficients.
\begin{thm} \label{t:main}
For any $\alpha \in (0,4)$, there exists a continuous function $u : (-\infty,0] \times \R^2 \to \R$ such that
\begin{itemize}
 \item The function $u$ is parabolic-homogeneous of degree $\alpha$. In other words, for all $a>0$,
 \[ u(t,x) = a^{-\alpha} u(a^2t,ax).\]
 \item The function $u$ is analytic in $((-\infty,0] \times \R^2) \setminus \{(0,0)\}$ and satisfies an equation of the form \eqref{e:measurable_coefficients} for some uniformly elliptic coefficients $a_{ij}(t,x)$.
\end{itemize} 
\end{thm}
Theorem \ref{t:main} shows that the H\"older continuity regularity obtained by Krylov and Safonov in \cite{ks} is not improvable for parabolic equations in dimension two or more, even if we impose regularity on the initial data. Note that unlike Safonov's example, our construction can be done \textbf{in dimension two}. This is a stark difference with respect to the elliptic case. Indeed, there is an old result by Nirenberg in \cite{nirenberg} proving that uniformly elliptic equations in 2D are always $C^{1+\alpha}$ regular, for some $\alpha>0$. Safonov's example takes advantage of a purely three dimensional geometric construction. Our example for Theorem \ref{t:main} is three dimensional in space-time.

Our second main result is about a solution to a fully nonlinear parabolic equation like \eqref{e:fully_nonlinear}. We present a solution which has an isolated singularity in space-time. Our justification has some part that is verified by a numerical computation. Because of that, it may be inappropriate to call it a \emph{Theorem}. Still, we believe our verification is sufficiently covincing so that we can confidently state it as a true fact.

\begin{fact} \label{f:main2}
There exists a function $u : Q_1 \to \R$ such that $u$ is analytic everywhere except at $(0,0)$. It solves a fully nonlinear uniformly parabolic equation like \eqref{e:fully_nonlinear}, and yet $u$ is not second differentiable in space at $(0,0)$.

Here $Q_1 := (-1,0] \times B_1$.
\end{fact} 

Unlike every example constructed so far for elliptic equations, the function $u$ in (almost)-Theorem \ref{f:main2} is not homogeneous. In fact, there is no singular solution $u$ that is homogeneous in $x$ on $t=0$. This is arguably unexpected, so we presented it as our third main result.

\begin{thm} \label{t:impossibility}
For any dimension, there exists no function $u$ solving the equation \eqref{e:fully_nonlinear} in $Q_1$ so that $u(0,x)$ is homogeneous of degree two in $x$, $u$ is analytic in $Q_1 \setminus (0,0)$, but $u$ is not $C^2$ at the origin.
\end{thm}

There are some non-existence results for elliptic equations in the literature that are worth mentioning, and comparing them with Theorem \ref{t:impossibility}. 

There is a result by Han, Nadirashvili and Yuan \cite{HanNadirashviliYuan} proving that there exists no singular solution of \eqref{e:elliptic_measurable} that is homogeneous of degree one. Note that the examples that Safonov constructs in \cite{safonov} are homogeneous of degree strictly below one. Our examples in Theorem \ref{t:main} can be homogeneous of degree one or even larger (but for parabolic equations).

Thanks to the result in \cite{nadirashvili2013}, we know that there exist no singular solution to a fully nonlinear elliptic equation of the form $F(D^2u) = 0$, homogeneous of degree two and analytic away from the origin, in dimensions four or less. Even though this result shows there is a serious obstruction  to build a singular solution similar to that of \cite{NVT}, it does not rule out the existence of singular solutions that are homogeneous of a smaller degree, or not homogeneous at all.

It is also worth mentioning the result in \cite{NYuan} saying that a homogeneous solution of a fully nonlinear elliptic equation, of any degree other than two, in any dimension, must be a polynomial if $F \in C^1$.

In analogy with the elliptic constructions in \cite{nadirashvili2013} and the proof of Theorem \ref{t:main}, it would seem natural to try to build a parabolic-homogeneous solution of a fully nonlinear parabolic equation of the form \eqref{e:fully_nonlinear}. That is, a function $u : (-\infty,0] \times \R^d \to \R$ such that
\[ u(t,x) = \lambda^{-1-\alpha} u(\lambda^2 t,\lambda x) \text{ for all } \lambda >0. \]
Such a function would necessarily be $C^{1+\alpha}$ in space but not $C^{2+\alpha}$ on $t=0$. However, there is no function of this form that solves a fully nonlinear parabolic equation since its time derivative $u_t$ would fail to be bounded around the origin. It is well known that the time derivative of a solution to a fully nonlinear parabolic equation must be H\"older continuous, and in particular locally bounded. Singular solutions to a fully nonlinear parabolic equation as in (almost)-Theorem \ref{f:main2} cannot be parabolic-homogeneous. 

Every solution to a fully nonlinear parabolic equation is also a solution to an elliptic equation with right hand side $u_t$, for every fixed value of $t$. Since $u_t$ is H\"older continuous, any singular solution to a parabolic equation must agree with a singular solution to an elliptic equation with a H\"older continuous right hand side at the final time. After this observation, it would be natural to attempt to build such a singular solution by making $u(0,x)$ equal to some of the known examples of singular solutions to elliptic equations, for example the one from \cite{NVT}. However, all these known examples are homogeneous in $x$. Theorem \ref{t:impossibility} rules out any function of this form. Its proof involves an analysis of the time derivative $u_t$, but it is more subtle than the analysis above for parabolic-homogeneous functions.

In Section \ref{s:preliminaries} we present a few lemmas that characterize when a function $u$ solves an equation of the form \eqref{e:measurable_coefficients} for \emph{some} uniformly elliptic coefficients. In order to prove Theorem \ref{t:main}, we show a simple explicit formula for the function $u$ that satisfies the criteria established in Section \ref{s:preliminaries}. The main difficulty of proving Theorem \ref{t:main} is in finding the right function $u$. Once the explicit formula for $u$ is exposed, it is admittedly easy to verify it satisfies an equation like \eqref{e:measurable_coefficients}. 

The justification of (almost)-Theorem \ref{f:main2} is given in Section \ref{s:fully_nonlinear}. We prove a lemma characterizing the functions $u$ that solve some fully nonlinear parabolic equation. Then we write an explicit function that satisfies that condition. The verification of the condition is done numerically with the help of a computer. Because of that, it is not a complete analytical proof, but it is very convincing.

The proof of Theorem \ref{t:impossibility} is given in Section \ref{s:impossibility}. The idea is to use the homogeneity assumption together with the $C^\alpha$ estimates to prove that $u_t(0,x)$ is constant. The time derivative $u_t$ solves a uniformly parabolic equation with coefficients depending on $D^2 u$. A unique continuation result gives us backward uniqueness for that equation from which we determine that $u_t$ is constant everywhere in $Q_1$ and that leads to the proof.

The functions that realize the examples in Theorems \ref{t:main} and \ref{f:main2} are explicit. For Theorem \ref{t:main} it is
\[ u(t,x) = \frac{|x|^2+t}{(|x|^2-t)^{1-\alpha/2}}. \]
For (almost)-Theorem \ref{f:main2}, it is
\[ u(t,x) = \frac{P_5(x)}{-t + \sqrt{|x|^2+t^2}} + \frac 1 {12} P_5(x).\]
Here, $P_5$ is the isoperimetric Cartan cubic polynomial in dimension five, used in \cite{NVT}, given by the formula
\[ P_5(x) = x_1^3 + \frac 3 2 x_1 (x_3^2+x_4^2-2x_5^2-2x_2^2) + \frac{3\sqrt 3} 2 (x_2 x_3^2 - x_2 x_4^2 + 2 x_3 x_4 x_5) .\]

The equations are not explicit. We do not compute the coefficients $a_{ij}$ for Theorem \ref{t:main}, or the function $F$ for (almost)-Theorem \ref{f:main2}.

\begin{remark}
For historical reasons, we use the term \emph{measurable coefficients} to refer to an equation of the form \eqref{e:measurable_coefficients}. It is important that the coefficients $a_{ij}$ are uniformly elliptic but they do not satisfy any further continuity assumption. The measurability of these coefficients is largely irrelevant.

The equation is restated in terms of the Pucci operators in Section \ref{s:preliminaries}. Note that since the equation is in non-divergence form, the solutions cannot be understood in the sense of distributions. An appropriate way to make sense of whether a non-smooth function $u$ solves an equation like \eqref{e:measurable_coefficients} for some uniformly elliptic coefficients $a_{ij}$ is given by two inequalities that must be satisfied in the viscosity sense. The coefficients $a_{ij}$ are implicit, their pointwise values are not necessarily well defined.
\end{remark}

\subsection{Complementary results}

We include a few extra results that complement our main results above. They answer some natural related questions and they are proved through similar techniques.

Even though Theorem \ref{t:main} suggests that there might exist singular solutions to uniformly elliptic fully nonlinear parabolic equations in two dimensions, that is not the case. We state that fact in the first complementary result.

\begin{prop} \label{p:2d}
Let $B_1$ be the unit ball in $\R^2$ and $u :  Q_1 \to \R$ be a viscosity solution to a fully nonlinear equation of the form
\[ u_t - F(D^2 u) = 0 \text{ in } Q_1.\]
Assume $F$ is uniformly elliptic and $f \in C^\alpha(Q_1)$. Then $u$ is $C^{2+\alpha}$ in space and $C^{1+\alpha/2}$ in time for some $\alpha>0$. Moreover, it satisfies the estimates
\[ \|D^2_x u\|_{C^\alpha(Q_{1/2})} \leq C \|u\|_{C^0(Q_1)}.\]
\end{prop}
We use the notation $Q_r = (-r^2,0]\times B_r$ to denote the parabolic cylinder centered at the origin.

The proof of Proposition \ref{p:2d} is simpler than it may seem at first sight. It is based on the following observation. The time derivative of the solution to a uniformly parabolic equation is known to be H\"older continuous in any dimension. Thus, the fully nonlinear parabolic equation can be thought of as an elliptic one, for every frozen value of $t$, with a H\"older continuous right hand side. Fully nonlinear elliptic equations in 2D are known to have $C^{2+\alpha}$ solutions. The proof of Proposition \ref{p:2d} follows then by combining these well known tools. Given the simplicity of its proof, it is difficult to assert whether Proposition \ref{p:2d} is new or not. We could not find it in the literature. We include it in this paper because it complements our first main result: Theorem \ref{t:main}.

We are able to extend Theorem \ref{t:impossibility} to other degrees of homogeneity provided that $F$ is homogeneous of degree one.

\begin{prop} \label{p:impossibility_homogeneous}
If we assume futher than $F$ is homogeneous of degree one, then there exists no function $u$ solving an equation like \eqref{e:fully_nonlinear}, in any dimension, so that $u(0,x)$ is homogeneous of degree less than two in $x$, and $u$ is analytic in $Q_1 \setminus (0,0)$.
\end{prop}

The analyticity assumption on $u$ is also removable at the expense of further smoothness assumptions on $F$ and having $u$ globally defined.

\begin{prop} \label{p:impossibility_C11}
If we assume futher than $F$ a $C^{1,1}$ function, then there exists no function $u : (-1,0] \times \R^d \to \R$, in any dimension, $C^3$ away from $(0,0)$, solving an equation like \eqref{e:fully_nonlinear}, so that $u(0,x)$ is homogeneous of degree two in $x$, it is not $C^2$ at the origin, and $|(\partial_t,\nabla_x) D_x^2 u| \lesssim |x|^{-1}$.
\end{prop}

For the proof of Proposition \ref{p:impossibility_C11}, we replace the analyticity condition in Theorem \ref{t:impossibility} with a unique continuation argument for parabolic equations. The assumption that $F \in C^{1,1}$ arises in similarly as in \cite{scott2011} as a way to obtain a linearized equation with Lipchitz coefficients. The proof of Proposition \ref{p:impossibility_C11} relies on the backward uniqueness result from \cite{WuZhang}. Note that this type of uniqueness results for parabolic equations would not apply to solutions in a bounded domain.

\subsection*{Acknowledgments}

The author is supported by NSF grant DMS-1764285.

The computational component in the verification of (almost)-Theorem \ref{f:main2} was completed in part with resources provided by the University of Chicago Research Computing Center.

\section{Preliminaries}
\label{s:preliminaries}

In this section we review some standard notions for elliptic and parabolic equations in nondivergence form. None of the statements in this section is new.

Given any two ellipticity constants $\Lambda \geq \lambda >0$, we define the usual Pucci operators over the set of real symmetric matrices.
\begin{align*} 
	P^+(M) &:= \Lambda \tr M_+ - \lambda \tr M_-, \\
	P^-(M) &:= \lambda \tr M_+ - \Lambda \tr M_-.
\end{align*}

Here, we write $M_+$ and $M_-$ to represent the positive and negative parts of the symmetric matrix $M$. We use the convention that both $M_+$ and $M_-$ are positive definite. Thus, $M_-=-M$ when $M$ is negative definite.

The Pucci operators represent the extremal elliptic operators with ellipticity constants $\lambda$ and $\Lambda$. In fact, the following identity holds.
\begin{align} 
	P^+(M) &= \sup \left\{\tr (AM) : \lambda \mathrm{I} \leq A \leq \Lambda \mathrm{I} \right\}, \label{e:Pmin}\\
	P^-(M) &= \inf \left\{\tr (AM) : \lambda \mathrm{I} \leq A \leq \Lambda \mathrm{I} \right\}. \label{e:Pmax}
\end{align}
Identities \eqref{e:Pmin} and \eqref{e:Pmax} are very well known. Probably because of their simplicity, their proofs are most often omitted. Let us write a quick justification. The first thing to notice in order to prove \eqref{e:Pmin} is that it holds when $M$ is positive definite. Indeed, in that case we have 
$\tr AM = \tr M^{1/2} A M^{1/2}$ and $ \lambda M = M^{1/2} \lambda M^{1/2} \leq M^{1/2} A M^{1/2} \leq M^{1/2} \Lambda M^{1/2} = \Lambda M$. Thus, $\lambda \tr M \leq \tr AM \leq \Lambda M$ when $M$ is positive definite. If $M$ is not positive definite, we write $M = M_+ - M_-$, use the previous inequality for $M_+$ and $M_-$ and the linearity of the trace to get
\[ \tr AM = \tr AM_+ - \tr AM_- \leq \Lambda \tr M_+ - \lambda \tr M_- = P^+(M).\]
Similarly, we also get $\tr AM \geq P^-(M)$. The equality in \eqref{e:Pmin} is achieved when $A$ equals $\Lambda$ times the projector over the positive eigenvalues, plus $\lambda$ times the projector over its orthogonal complement. The equality in \eqref{e:Pmax} holds with the opposite choice of constants.

From the characterization of $P^+$ and $P^-$ as in (\ref{e:Pmin}-\ref{e:Pmax}), we deduce the following corollary.
\begin{cor} \label{c:pucci}
Given any symmetric matrix $M \in \R^{d \times d}$ the following two statements are equivalent
\begin{enumerate}
	\item There exists a symmetric matrix $A \in \R^{d \times d}$, so that $\lambda \mathrm{I} \leq A \leq \Lambda \mathrm{I}$, and $\tr A M = 0$.
	\item $P^+(M) \geq 0 \geq P^-(M)$.
\end{enumerate}
\end{cor}

There are different ways to consider non-smooth solutions to an equation like \eqref{e:measurable_coefficients}. One possibility is to consider a function $u$, whose derivatives in the sense of distributions $u_t$ and $\partial_{ij} u$ make sense as functions, at least in $L^1_{loc}$, and solve \eqref{e:measurable_coefficients} for some uniformly elliptic measurable coefficients $a_{ij}(t,x)$. Without any further regularity assumption on $u$, this notion of solution in the sense of distributions has some severe shotcomings (lack of existence and lack of uniqueness among others). It is most convenient to reformulate the equation \eqref{e:measurable_coefficients} in the viscosity sense using the Pucci operators. From Corollary \ref{c:pucci}, we see that a smooth function $u$ solves \eqref{e:measurable_coefficients} for \emph{some} uniformly elliptic coefficient $a_{ij}(t,x)$ if and only if
\begin{equation} \label{e:pucci}
u_t - P^+(D^2u) \leq 0 \text{ and } u_t - P^-(D^2u) \geq 0.
\end{equation}
The equation \eqref{e:measurable_coefficients} with undetermined rough coefficients $a_{ij}(t,x)$, turns out to be equivalent to the pair of nonlinear inequalities \eqref{e:pucci}. An advantage of \eqref{e:pucci} is that we can make sense of the inequalities in the viscosity sense, for functions $u$ that are merely continuous. The example we construct in this paper to verify Theorem \ref{t:main} is a function $u$ that remains smooth up to the point of the singularity. In that sense, it is not necessary for us to work with any generalized notion of solution. Yet, we verify the two inequalities in \eqref{e:pucci} rather than constructing the coefficients $a_{ij}(t,x)$ explicitly.

It is also common and useful to restate the notion of uniform ellipticity for a nonlinear function $F$ in terms of the Pucci operators. Indeed, the function $F$ is uniformly elliptic with parameters $\lambda$ and $\Lambda$ if and only if for any pair of symmetric matrices $A,B \in \R^{n \times n}$, it holds
\[ P^-(B) \leq F(A+B) - F(A) \leq P^+(B).\]
In this case, the inequalities hold for any two symmetric matrices $A$ and $B$, without requiring $B$ to be symmetric. It is easy to verify that the Pucci operators $P^+$ and $P^-$ are themselves uniformly elliptic with the same constants $\lambda$ and $\Lambda$.

\subsection{Review on regularity results for parabolic equations}

The following two theorems summarize the fundamental regularity results that follows from the work of Krylov and Safonov \cite{ks}. In the context of viscosity solutions, we also reference \cite{lihe1}, \cite{lihe2}, and the lecture notes \cite{MR3185332}.

\begin{thm} \label{t:ks1}
Let $u$ be a continuous function that satisfies the two inequalities \eqref{e:pucci} in the viscosity sense in $Q_1$. Then, for some $\alpha>0$ small depending on $\lambda$, $\Lambda$ and dimension only, $u \in C^\alpha(Q_{1/2})$ and
\[ \|u\|_{C^\alpha(Q_{1/2})} \leq C \|u\|_{C^0(Q_1)}.\]
The constant $C$ depends also on $\lambda$, $\Lambda$ and dimension only.
\end{thm}

Applying Theorem \eqref{t:ks1} to incremental quotients of a solution to a fully nonlinear parabolic equation like \eqref{e:fully_nonlinear}, we obtain the following $C^{1+\alpha}$ estimate.

\begin{thm} \label{t:ks2}
Let $u$ be a continuous function that satisfies \eqref{e:fully_nonlinear} in the viscosity sense in $Q_1$. Assume that $F$ is uniformly elliptic. Then, for some $\alpha>0$ small depending on $\lambda$, $\Lambda$ and dimension only, $u \in C^{1+\alpha}(Q_{1/2})$ and
\[ \|(\partial_t u, \nabla u)\|_{C^\alpha(Q_{1/2})} \leq C \|u\|_{C^0(Q_1)}.\]
The constant $C$ depends also on $\lambda$, $\Lambda$ and dimension only.
\end{thm}

The $C^\alpha$ estimate on the time derivative $u_t$ plays an important role in the proof of Theorem \ref{t:impossibility}. In \cite{changlara}, they prove that the estimate on $u_t$ holds even if we add a $C^\alpha$ right hand side to the equation.

The last theorem we want to reference in this section is a $C^{2+\alpha}$ estimate for fully nonlinear elliptic equations in 2D with a H\"older continuous right hand side.

\begin{thm} \label{t:nirenberg-caffarelli}
Let $u : B_1 \to \R$ be a viscosity solution of
\[ F(D^2 u) = f(x) \text{ in } B_1.\]
Here $B_1$ is the unit ball in $\R^2$ and $F$ is uniformly elliptic. Then, for some $\alpha>0$ small depending on $\lambda$, $\Lambda$ and dimension only, $u \in C^{2+\alpha}(B_{1/2})$ and
\[ \|u\|_{C^{2+\alpha}(B_{1/2})} \leq C \left( \|u\|_{C^0(B_1)} + \|f\|_{C^\alpha(B_1)} \right).\]
The constant $C$ depends also on $\lambda$, $\Lambda$ and dimension only.
\end{thm}

The $C^{2+\alpha}$ estimate of Theorem \ref{t:nirenberg-caffarelli} is originally due to Nirenberg \cite{nirenberg} when the right hand side is zero. Caffarelli's Schauder estimates for fully nonlinear elliptic equations (see \cite[Section 8.1]{CC}) allow us to add a H\"older continuous right hand side to essentially any elliptic equation that satisfies a $C^{2+\alpha}$ estimate.

Naturally, the smaller cylinder $Q_{1/2}$ can be replaced with $Q_\rho$ in Theorems \ref{t:ks1}, \ref{t:ks2} and \ref{t:nirenberg-caffarelli}, for any value of $\rho \in (0,1)$, by adjusting the constants $C$.

\section{Singular solutions to parabolic equations with measurable coefficients in 2D}

Theorem \ref{t:main} is justified by the following explicit function
\begin{equation} \label{e:u}
 u(t,x) = \frac{|x|^2+t}{(|x|^2-t)^{1-\alpha/2}}.
\end{equation}

Once we know the function $u$ explicitly, the proof of Theorem \ref{t:main} is a relatively short computation to verify that \eqref{e:pucci} holds.

\begin{proof}[Proof of Theorem \ref{t:main}]
We have to verify that the function $u$ given in \eqref{e:u} satisfies \eqref{e:pucci}. It is clear that this function is smooth away from $(0,0)$ and that it has the desired homogeneity.

The function $u$ is radially symmetric with respect to $x$, so the eigenvectors of $D^2 u$ are in the radial direction and its perpendicular. Let us use polar coordinates and call $r = |x|$. The eigenvalues of $D^2 u$ are precisely $u_{rr}$ and $u_r/r$. The following are the values of $u_t$ and $u_r/r$, obtained via a direct computation.

\begin{align*}
u_t &= (r^2-t)^{-2+\alpha/2} \left( \left( 2 - \alpha/2 \right) r^2 - \frac \alpha 2 t \right), \\
\frac {u_r} r &= (r^2-t)^{-2+\alpha/2} \left( \alpha r^2 - (4 - \alpha) t \right).
\end{align*}

We observe that wherever $t<0$, we have both $u_r/r > 0$ and $u_t > 0$. The value of $u_{rr}$ may have either sign and we do not need to compute it explicitly.

The function $u$ is parabolic-homogeneous of degree $\alpha$. That means that $\lambda^\alpha u(t,x) = u(\lambda^2 t, \lambda x)$. Therefore, we also have
\[ u_t(\lambda^2 t, \lambda x) - P^\pm(D^2u(\lambda^2 t, \lambda x)) = \lambda^{\alpha-2} \left( u_t(t,x) - P^\pm(D^2u(t,x)) \right).   \]
Thus, it is enough to verify the hypothesis \eqref{e:pucci} on the surface $\T := \{|x|^2 - t = 1\} \cap \{t\leq 0\}$.

Note that $\mathcal S$ is a compact surface. The functions $u_t$ and $u_r/r$ achieve their maximum and minimum positive values. In order to verify \eqref{e:pucci} on $\T$, we pick $\lambda$ and $\Lambda$ such that
\begin{align*}
\max_\T u_t - \Lambda \min_\T (u_r/r) + \lambda \max_\T u_{rr}^- \leq 0, \\
\min_\T u_t - \lambda \max_\T (u_r/r) - \lambda \max_\T u_{rr}^+ \geq 0.
\end{align*}
First, using that $\min_\T u_t>0$, we pick $\lambda>0$ small enough to ensure that the second inequality holds. Then, using that $\min_\T (u_r/r)>0$, pick $\Lambda$ large enough to ensure that the first inequality holds. These two inequalities ensure that \eqref{e:pucci} holds everywhere on $\T$. Because of the homogeneity of $u$, \eqref{e:pucci} holds everywhere, which concludes the proof.
\end{proof}

\section{Fully nonlinear parabolic equations in 2D.}

In this section, we prove Proposition \ref{p:2d}

From Theorem \ref{t:ks2}, we know that any viscosity solution $u$ to a fully nonlinear parabolic equation, in any dimension, is differentiable in time and satisfies, for any $\rho < 1$.
\[ \|u_t\|_{C^\alpha(Q_\rho)} \lesssim \|u\|_{C^0(Q_1)}. \]

Thus, given any solution to the parabolic equation $u_t - F(D^2u) = 0$, we can freeze time and consider the elliptic equation
\[ F(D^2u) = u_t. \]
Here, we think of $u_t$ as the right hand side of an elliptic equation. For any $t$, the function $u(t,\cdot)$ solves a uniformly elliptic equation with a $C^\alpha$ right hand side. Applying Theorem \ref{t:nirenberg-caffarelli}, we have that $D^2u(t,\cdot)$ exists and is H\"older continuous for every fixed value of $t$. Moreover,
\[ \|D^2u(t,\cdot)\|_{C^\alpha(B_{3/4})} \lesssim \|u(t,\cdot)\|_{C^0(B_{7/8})} + \|u_t(t,\cdot)\|_{C^\alpha(B_{7/8})} \lesssim \|u\|_{C^0(Q_1)}. \]
From here, we have the existence and H\"older continuity in space of $D^2 u$. We are only left to establish its H\"older continuity in time.

It is convenient to state the H\"older continuity with respect to the parabolic distance. In this case, since $\alpha$ is small, and $Q_{1/2}$ has a fixed size, the result would be equivalent to its H\"oder continuity with respect to the Euclidean distance in space-time. The parabolic distance is scale invariant with respect to the parabolic scaling. Because of that, it is the most appropriate distance when working with parabolic equations. It is
\[ d_p((t,x),(s,y)) := |x-y| + \sqrt{|s-t|}.\]
Using that $u_t$ is H\"older continuous with respect to the parabolic distance $d_p$, and that $u(t,\cdot) \in C^{2+\alpha}$ for every fixed value of $t$, we will prove that $D^2u$ is H\"older continuous in space-time with respect to the parabolic distance.

Let $(t,x)$ and $(s,x)$ be two points in $Q_{1/2}$ so that $|t-s| < r^2$. Let us analyze the values of $u$ on $\{t\} \times B_r(x)$ and $\{s\} \times B_r(x)$.

Using the $C^\alpha$ regularity of $u_t$, with respect to the parabolic distance, we get, for any $y \in B_r(x)$
\[ |u(t,y) - u(s,y) - (t-s) u_t(t,y)| \leq C r^{2+\alpha}. \]
Moreover, since $|u_t(t,y)-u_t(t,x)|\leq C r^\alpha$, we also get
\[ |u(t,y) - u(s,y) - (t-s) u_t(t,x)| \leq C r^{2+\alpha}. \]

Let us know compare their second order Taylor expansions in the space variable. Since $u$ is $C^{2+\alpha}$ in space, we have
\[ |u(t,y) - u(t,x) - (y-x) \cdot \nabla u(t,x) - \frac 12 (y_i-x_i)(y_j-x_j) \partial_{ij} u(t,x)| \leq C r^{2+\alpha}.\]
\[ |u(s,y) - u(s,x) - (y-x) \cdot \nabla u(s,x) - \frac 12 (y_i-x_i)(y_j-x_j) \partial_{ij} u(s,x)| \leq C r^{2+\alpha}.\]
Adding the three inequalities above, we deduce an inequality for a second order polynomial that says
\[ \sup_{z \in B_r} |a_{ij} z_i z_j + b \cdot z + c| \leq C r^{2+\alpha},\]
where, $z$ stands for $y-x$, and the coefficients of the polynomial are
\begin{align*}
a_{ij} &= \partial_{ij} u(t,x) - \partial_{ij} u(s,x), \\
b &= \nabla u(t,x)-\nabla u(s,x), \\
c &= u(s,x)-u(t,x) + (t-s) u_t(t,x).
\end{align*}
The sup norm on the space of second order polynomials is equivalent to any other norm since it is a finite dimensional space. In this case, taking the scaling into account, that means
\begin{align*}
a_{ij} &\leq C r^\alpha, \\
b &\leq C r^{1+\alpha}, \\
c &\leq C r^{2+\alpha}.
\end{align*}
Thus, we obtained that $|D^2u(t,x) - D^2u(s,x)| \leq C r^{\alpha}$. This is the H\"older continuity of $D^2u$ in time and we finished the proof of Proposition \ref{p:2d}.

\section{Singular solutions to Fully nonlinear parabolic equations}

\label{s:fully_nonlinear}

In this section, we explain how to verify if a function $u$ solves some equation of the form \eqref{e:fully_nonlinear}. We explain the justification of (almost)-Theorem \ref{f:main2}

We start with the following lemma, characterizing the functions that solve a uniformly parabolic fully nonlinear equation.

\begin{lemma} \label{l:fn_pucci}
Let $\Lambda \geq \lambda > 0$. Given a (space-time) set $\Omega \subset \R \times \R^d$ and a function $u : \Omega \to \R$, which is second differentiable in $x$ and differentiable in $t$, then the following two statements are equivalent.
\begin{enumerate}
	\item There exists a uniformly elliptic function $F$ so that $u_t - F(D^2 u) = 0$ holds in $\Omega$.
	\item For any pair of points $(t,x)$ and $(s,y)$ in $\Omega$, we have
	\begin{equation} \label{e:pucci_condition}
	P^-( D^2 u(t,x) - D^2 u(s,y)) \leq u_t(t,x) - u_t(s,y) \leq P^+( D^2 u(t,x) - D^2 u(s,y)).
	\end{equation}
	Here, $P^+$ and $P^-$ are the extremal Pucci operators defined in \eqref{e:Pmin} and \eqref{e:Pmax} with ellipticity parameters $\lambda$ and $\Lambda$.
\end{enumerate}
\end{lemma}

\begin{proof}
We start with the easy implication $(1) \Rightarrow (2)$. For any pair of points $(t,x)$, $(s,y)$ we know that the equation in (1) holds at both points. Using the ellipticity of $F$, we get that $P^-( D^2 u(t,x) - D^2 u(s,y)) \leq F(D^2 u(t,x)) - F(D^2 u(s,y)) \leq P^+( D^2 u(t,x) - D^2 u(s,y))$, from which 2. follows.

In order to prove $2. \Rightarrow 1.$, we need to construct a nonlinear function $F$ for any given function $u$ satisfying the inequalities in (2). Such function may not be unique. One possibility is given by the formula
\begin{equation} \label{e:F-reconstructed}
F(M) = \sup_{(t,x) \in \Omega} u_t(t,x) + P^-(M - D^2u(t,x)).
\end{equation}
Let us first argue that our definition makes sense. For that, we have to verify that the right hand side is finite. Let us fix any point $(t_0,x_0) \in \Omega$. Using the uniform ellipticity of $P^-$, we observe that for any other $(t,x) \in \Omega$ we have
\[ P^-(A - D^2u(t,x)) \leq P^+(A - D^2u(t_0,x_0)) + P^-(D^2u(t_0,x_0) - D^2u(t,x)).\]
Therefore, using 2. with $(s,y) = (t_0,x_0)$,
\[ u_t(t,x) + P^-(A - D^2u(t,x)) \leq u(t_0,x_0) + P^+(A - D^2u(t_0,x_0)). \]
This gives us a uniform bound for the right hand side of \eqref{e:F-reconstructed}. Thus, $F(M)$ is well defined and finite for every symmetric matrix $M$. We are left to verify that $F$ is uniformly elliptic.

Let $A$ and $B$ be any two symmetric matrices in $\R^{n \times n}$. Let us compare the values of $F(M+B)$ and $F(M)$. Using that the Pucci operator $P^-$ itself is uniformly elliptic, for any $(t,x) \in \Omega $we get
\[  u_t(t,x) + P^-(A - D^2u(t,x)) + P^-(A) \leq u_t(t,x) + P^-(A+B - D^2u(t,x)) u_t(t,x) + P^-(A - D^2u(t,x))  + P^+(B).\]
Therefore $F(A)+P^-(B) \leq F(A+B) \leq F(A)+P^+(B)$, thus $F$ is uniformly elliptic.
\end{proof}

The following is a similar characterization as in Lemma \ref{l:fn_pucci}. It is slightly easier to implement in a numerical computation.
\begin{cor} \label{c:fn_ratio}
Given a (space-time) set $\Omega \subset \R \times \R^d$ and a function $u : \Omega \to \R$, which is second differentiable in $x$ and differentiable in $t$, then the following two statements are equivalent.
\begin{enumerate}
	\item For some $\Lambda \geq \lambda > 0$, there exists a uniformly elliptic function $F$ so that $u_t - F(D^2 u) = 0$ holds in $\Omega$.
	\item There exists some constant $C \geq 1$, so that for any pair of points $(t,x)$ and $(s,y)$ in $\Omega$, we have
	\begin{equation} \label{e:ratio_condition}
	C^{-1} \leq \frac{(\partial_t u(t,x)-\partial_t u(s,y))_- + \mathrm{tr}(D^2u(t,x)-D^2u(s,y))_+}{(\partial_t u(t,x)-\partial_t u(s,y))_+ + \mathrm{tr}(D^2u(t,x)-D^2u(s,y))_-} \leq C.
	\end{equation}
	Here, we write $a_+$ and $a_-$ to denote the positive and negative part of a number or a symmetric matrix. We use the convention $a_- = -a$ if $a \leq 0$.
\end{enumerate}
\end{cor}

\begin{proof}
We prove that the second condition in Corollary \ref{c:fn_ratio} is equivalent to the second condition in Lemma \ref{l:fn_pucci}. Indeed, if (2) in Lemma \ref{l:fn_pucci} holds, we observe that (2) in Corollary \ref{c:fn_ratio} holds as well with $C = \max(\Lambda,\lambda^{-1},\Lambda / \lambda)$. Conversely, if (2) in Corollary \ref{c:fn_ratio} holds, then (2) in Lemma \ref{l:fn_pucci} holds as well with $\Lambda = C$ and $\lambda = C^{-1}$.
\end{proof}

The main strategy for finding singular solutions to some fully nonlinear parabolic equation is to write a candidate function and verify condition (2) in Corollary \ref{c:fn_ratio}. A similar strategy is used in \cite{NVT} to verify that their function solves a fully nonlinear elliptic equation. In that case, the homogeneity and some symmetries of the functions are used to simplify the computation and a fully analytical proof is given. We do not give a full analytical proof of Theorem \ref{f:main2} in this paper. Instead, we verify numerically that the condition from Corollary \ref{c:fn_ratio} holds for certain candidate function. Below, we describe our implementation of this verification.

A straight forward brute-force approach to verify the condition in Corollary \ref{c:fn_ratio} would be to sample a large number of random pairs of points and verify that they satisfy \eqref{e:ratio_condition}. If either the numerator or the denominator in \eqref{e:ratio_condition} vanishes in a large proportion of $Q_1$, we would identify some of these points quickly and rule out our candidate function. However, this naive algorithm is prone to false positives due to the \emph{curse of dimensionality}. There exist functions for which the numerator and denominator in \eqref{e:ratio_condition} only vanish on surfaces with a relatively high codimension. It is very difficult to randomly find a pair of points sufficiently near such a surface. For example, let us consider the function
\begin{equation} \label{e:candidate1}
u(t,x) = \frac{P_5(x)}{\sqrt{|x|^2-t}}.
\end{equation}
Here, $P_5$ is the cubic polynomial used in \cite{NVT}. If we sample a million pairs of random points in $Q_1$, in all likelihood, all of them would verify condition \eqref{e:ratio_condition} for a constant $C$ being approximately 15. However, this function is not a solution to a fully nonlinear parabolic equation since its time derivative is unbounded near zero. The condition \eqref{e:ratio_condition} is only invalidated when $|t-s|$ is small, and $|x-y|$ is much smaller. Since $x$ and $y$ are five dimensional vectors, it is very unlikely that a random selection of points will ever sample a pair where $x$ and $y$ are practically identical.

A possible (but arguably unnatural) workaround would be to start by testing that $u_t$, and all the partial derivatives $\partial_{x_i} u$, satisfy an equation with measurable coefficients like \eqref{e:measurable_coefficients}. This can be tested by checking that large sample of points verifies \eqref{e:pucci} for every one of those derivatives. The function above would not pass the test for $u_t$. However, this algorithm will fail to rule out other functions. For example, let
\begin{equation} \label{e:candidate2}
u(t,x) = \frac{P_5(x)}{-t + \sqrt{|x|^2+t^2}}.
\end{equation}
For this function, we have that $u_t$ satisfies \eqref{e:pucci}. The partial derivatives $\partial_{x_i} u$ do not, but they only fail on the line $t=0$ and $x = a e_i$, for $a \in \R$. This is a segment that has codimention 5 in $Q_1$. Again, it is very unlikely for a random sample of points to ever hit near a set of codimension five.

The algorithm that easily rules out both examples above is to follow a stochastic gradient flow for the ratio \eqref{e:ratio_condition} and verify that it stays bounded. Note that the gradient flow will tend to some local maximum for the ratio in \eqref{e:ratio_condition}, and there may be many of them. If $u$ does not satisfy \eqref{e:ratio_condition}, the gradient flow may or may not diverge depending on the choice of the initial point. Thus, we still have to sample several possible random initial points and start our stochastic gradient flow from each one of them.

In our test, we see the ratio diverge very quickly for the function \eqref{e:candidate1}. For the example \eqref{e:candidate2}, the gradient flow seems to diverge for approximately 2\% of the initial points.

According to our test, the following function satisfies \eqref{e:ratio_condition} for $C=14$ (the optimal value we get is approximately $C=13.7$).
\begin{equation} \label{e:candidate3}
u(t,x) = \frac{P_5(x)}{-t + \sqrt{|x|^2+t^2}} + \frac 1 {12} P_5(x).
\end{equation} 
We ran a large test by performing a stochastic gradient flow (with two million iterations) starting from a collection of several thousand random initial pairs of points in $Q_1$. The computation was carried out in the University of Chicago Research Computing Center. This is the function that verifies (almost)-Theorem \ref{f:main2}.

Note that the quotient in \eqref{e:ratio_condition} is discontinuous on $(s,y)=(t,x)$. The limits in each direction correspond to directional derivatives of $u$. In fact, according to our computations, the maximum value of the quotient is achieved near the diagonal $(s,y)=(t,x)$.

Even though this numerical computation cannot be considered a rigorous proof, it seems very convincing to us. Our algorithm could only fail in the unlikely scenario that the gradient flow for the quotient in \eqref{e:ratio_condition} for this function $u$, diverges only for some tiny proportion of the initial points.

For those who may want to perform numerical experiments themselves, we posted our source code at \url{https://math.uchicago.edu/~luis/singular_parabolic/sp.html}

The formula \eqref{e:candidate3} was clearly constructed by understanding where \eqref{e:candidate1} and \eqref{e:candidate2} fail to verify \eqref{e:ratio_condition} and modifying the functions accordingly. At this point, it is arguably not worth recounting all the other candidates that we tested and the reasons why they failed. It is interesting to point out that understanding why \eqref{e:candidate2} fails to solve a fully nonlinear parabolic equations is what motivated the proof of Theorem \ref{t:impossibility} below.

\begin{remark}
There is a curious fact that we observed applying our algorithm to verify some of the examples that we already knew for singular solutions to uniformly elliptic fully nonlinear elliptic equations. The optimal constant $C$ in \eqref{e:ratio_condition} appears to be an integer in those cases.

For elliptic equation, the condition \eqref{e:ratio_condition} applied to a function $u : \Omega \to \R$ that depends on $x$
only says that there exists a constant $C \geq 1$ such that for all pairs $x,y \in \Omega$,
\[ C^{-1} \leq \frac{\tr (D^2u(x)-D^2u(y))_+}{\tr (D^2u(x)-D^2u(y))_-} \leq C.\]
For each such solution $u$, there is a smallest value of the constant $C$ that makes the condition hold. This would be the \emph{optimal} value of $C$. If we let $u$ be the homogeneous of degree two function in five dimensions constructed in \cite{NVT}, our computation (with the algorithm described above) tells us that the optimal value of $C$ that makes the condition hold is $C=9$. The fact that it is exactly an integer number suggests that there should be a clean way to compute it. From the proofs in \cite{NVT}, one can deduce an upper bound for the optimal $C$, but we do not know any way to compute its exact value.

We learned the following (unpublished) example in nine dimensions from Charles Smart. If we let
\[ u(x_1,\dots,x_9) = \frac{1}{|x|} \det \begin{pmatrix} x_1 & x_2 & x_3 \\ x_4 & x_5 & x_6 \\ x_7 & x_8 & x_9 \end{pmatrix},\]
then it also solves some fully nonlinear elliptic PDE. In this case, the smallest value of the constant $C$ is thirteen (13), also an integer number.
\end{remark}

\section{Impossibility of singular homogeneous solutions for fully nonlinear parabolic equations}

\label{s:impossibility}

In this last section, we prove Theorem \ref{t:impossibility} and the related complementary results of Propositions \ref{p:impossibility_homogeneous} and \ref{p:impossibility_C11}.

The strategy of the proof is to use the homogeneity of the equation and the $C^\alpha$ regularity of $u_t$ to deduce that $u_t$ is constant at the final time $t=0$. Then, since $u_t$ satisfies a parabolic equation of the form \eqref{e:measurable_coefficients}, we use a unique continuation argument to prove that $u_t$ is constant everywhere. This means that any singularity at the final time $t=0$ would be propagated backwards to any earlier time.

\begin{lemma} \label{l:utconstant}
Let $u$ be a solution to a fully nonlinear parabolic equation
\[ u_t - F(D^2u) = 0 \text{ in } Q_1.\]
Assume that $F$ is uniformly elliptic, $u(0,x)$ is homogeneous of degree two in $x$ and $C^2$ away from the origin. Then $u_t(0,x)$ is constant for all $x \in B_1$.
\end{lemma}

\begin{proof}
Since $u(0,x)$ is homogeneous of degree two, then $D^2u(0,x)$ is homogeneous of degree zero. In other words $D^2u(0,ax) = D^2 u(0,x)$ for any $a>0$.

From the equation, we have that $u_t(0,x) = F(D^2u(0,x))$, therefore $u_t(0,x)$ is also constant along every ray emanating from the origin. 

From Theorem \ref{t:ks2}, we know that $u_t$ is H\"older continuous. However, the only homogeneous functions of degree zero that are continuous are the constant ones. So, $u_t(0,x)$ must be constant.
\end{proof}

\begin{lemma} \label{l:utsignchanging}
Let $u$ be a solution to a fully nonlinear parabolic equation
\[ u_t - F(D^2u) = 0 \text{ in } Q_1.\]
Assume that $u_t(0,x)=0$ for all $x \in B_1$. Then, for every $x_0 \in B_1$ and every $r>0$, the function $u_t$ is identically zero in $(-r,0] \times B_r(x_0)$ or it takes both positive and negative values there.
\end{lemma}

\begin{proof}
Assume that $u_t$ does not change sign in a neighborhood of $(0,x_0)$. That is, $u_t : (-r,0] \times B_r(x_0) \to \R$ is either non negative or non positive everywhere.

We know that $u_t$ satisfies a uniformly parabolic equation with measurable coefficients. Since $u(0,x_0)=0$, the strong maximum principle implies that $u$ must be identically zero.

Note that the strong maximum principle for parabolic equations is a direct consequence of the Harnack inequality. See \cite[Theorem Theorem 2.4.32]{MR3185332} and \cite[Section 7.1.4]{evans}.
\end{proof}

\begin{proof}[Proof of Theorem \ref{t:impossibility}]
From Lemma \ref{l:utconstant}, we know that $u_t(0,x)$ is constant for $x \in B_1$. Without loss of generality, we can assume that $u_t(0,x)=0$. Otherwise, we would consider the function $\tilde u(t,x) = u(t,x) - ct$. Clearly, $u$ satisfies \eqref{e:pucci_condition} if and only if $\tilde u$ does.

From Lemma \ref{l:utsignchanging}, we know that for every $x_0 \in B_1$ and for every $r>0$, the function $u_t$ is either constant or sign changing in $(-r,0] \times B_r(x_0)$. Using that $u$ is analytic, we will show that $u_t$ must be constant zero everywhere.

Let us pick any $x_0 \in B_1 \setminus \{0\}$. We claim that $\partial_t^k u(0,x)=0$ for all $x$ in a neighborhood of $x_0$. Let us assume for the sake of a contradiction that was not the case. Let $k$ be the smallest positive integer so that $\partial_t^k u(0,x)$ is not identically zero for $x$ is some neighborhood of $x_0$. We know that $u_t(0,x) \equiv 0$, so $k \geq 2$. Thus, there exists $r>0$ so that $\partial_t^j u(0,x)=0$ for all $x \in B_r(x_0)$ and $j = 1,\dots,k-1$. However, there is also an $x_1 \in B_1(x_0)$ so that $\partial_t^k u(0,x_1) \neq 0$. From continuity, we know that $\partial_t^k u$ does not change sign in some neighborhood of $(0,x_1)$. Therefore, for any $(t,x)$ in that neighborhood of $(0,x_1)$, the function $\partial_t u(t,x)$ will have the same sign as $(-1)^{k-1} \partial_t^k u(0,x_1)$ whenever $t<0$. This contradicts Lemma \ref{l:utsignchanging}. The contradiction comes from the existence of $k$. Therefore $\partial_t^k u(0,x)$ is identically zero for all $x \in B_1 \setminus \{0\}$ and all $k = 1,2,3,\dots$. From the analyticity of $u$ away from $(0,0)$ (with respect to time) we conclude that $u_t$ is identically zero in $Q_1$.

This means that $u(t,x)$ is constant with respect to $t$. So, it cannot fail to be $C^2$ at $(0,0)$ without having a singularity at $(t,0)$ for all $t < 0$.
\end{proof}

The proof of Proposition \ref{p:impossibility_homogeneous} follows similarly, replacing Lemma \ref{l:utconstant} by the following lemma

\begin{lemma} \label{l:utconstant_homogeneous}
Let $u$ be a solution to a fully nonlinear parabolic equation
\[ u_t - F(D^2u) = 0 \text{ in } Q_1.\]
Assume that $F$ is uniformly elliptic and homogeneouns of degree one, $u(0,x)$ is homogeneous of any degree smaller than two in $x$ and $C^2$ away from the origin. Then $u_t(0,x)=0$ for all $x \in B_1$.
\end{lemma}

\begin{proof}
Let us say that $u(0,x)$ is homogeneous of degree $\alpha$, for some $\alpha < 2$. Then $D^2u(0,x)$ is homogeneous of degree $\alpha-2$. In other words $D^2u(0,ax) = a^{\alpha-2} D^2 u(0,x)$ for any $a>0$.

From the equation, we have that $u_t(0,x) = F(D^2u(x,0))$. Therefore, using the homogeneity of $F$,
\[ u_t(0,ax) = F(D^2u(0,ax)) = F(a^{\alpha-2} D^2u(0,x)) = a^{\alpha-2} F(D^2u(0,x))  = a^{\alpha-2} u_t(0,x). \]
Thus, $u_t(0,x)$ is also homogenous of degree $\alpha-2 < 0$. The only way for such a function to be continuous at the origin is if it is identically zero.

From Theorem \ref{t:ks2}, $u_t$ is H\"older continuous. Then $u_t(0,x)$ must be zero for all $x \in B_1$.
\end{proof}

The proof of Proposition \ref{p:impossibility_C11} also proceeds along the same lines, but using a backwards uniqueness result for the equation satisfied for $u_t$ instead of its analyticity.

\begin{proof} [Proof of Proposition \ref{p:impossibility_C11}]
Like in the proof of Theorem \ref{t:impossibility}, we use Lemma \ref{l:utconstant} to conclude that $u_t(0,x)$ is constant for $x \in \R^d$. Withour loss of generality, we assert that $u_t(0,x)=0$ for all $x \in \R^d$.

Differentiating the equation with respect to $t$, we obtain the following equation for $u_t$,
\[ \partial_t (u_t) - a_{ij}(t,x) \partial_{ij} (u_t) = 0,\]
where $a_{ij}(t,x) = \partial F(D^2u) / \partial X_{ij}$.

From our assumptions, the coefficients $a_{ij}(t,x) = \partial F(D^2u) / \partial X_{ij}$ are Lipchitz and decay for large $|x|$ like in the conditions for the backward uniqueness result in \cite{WuZhang}. Therefore, $u_t \equiv 0$ everywhere, and we conclude in the proof of Theorem \ref{t:impossibility}.
\end{proof}

\bibliographystyle{plain}
\bibliography{sp}

\begin{thebibliography}{10}

\bibitem{scott2011}
Scott~N. Armstrong and Luis Silvestre.
\newblock Unique continuation for fully nonlinear elliptic equations.
\newblock {\em Math. Res. Lett.}, 18(5):921--926, 2011.

\bibitem{CC}
Luis~A. Caffarelli and Xavier Cabr\'{e}.
\newblock {\em Fully nonlinear elliptic equations}, volume~43 of {\em American
  Mathematical Society Colloquium Publications}.
\newblock American Mathematical Society, Providence, RI, 1995.

\bibitem{changlara}
H\'{e}ctor~A. Chang-Lara and Dennis Kriventsov.
\newblock Further time regularity for nonlocal, fully nonlinear parabolic
  equations.
\newblock {\em Comm. Pure Appl. Math.}, 70(5):950--977, 2017.

\bibitem{evans}
Lawrence~C. Evans.
\newblock {\em Partial differential equations}, volume~19 of {\em Graduate
  Studies in Mathematics}.
\newblock American Mathematical Society, Providence, RI, 1998.

\bibitem{HanNadirashviliYuan}
Qing Han, Nikolai Nadirashvili, and Yu~Yuan.
\newblock Linearity of homogeneous order-one solutions to elliptic equations in
  dimension three.
\newblock {\em Comm. Pure Appl. Math.}, 56(4):425--432, 2003.

\bibitem{MR3185332}
Cyril Imbert and Luis Silvestre.
\newblock An introduction to fully nonlinear parabolic equations.
\newblock In {\em An introduction to the {K}\"{a}hler-{R}icci flow}, volume
  2086 of {\em Lecture Notes in Math.}, pages 7--88. Springer, Cham, 2013.

\bibitem{ks}
N.~V. Krylov and M.~V. Safonov.
\newblock A property of the solutions of parabolic equations with measurable
  coefficients.
\newblock {\em Izv. Akad. Nauk SSSR Ser. Mat.}, 44(1):161--175, 239, 1980.

\bibitem{NVT}
Nikolai Nadirashvili, Vladimir Tkachev, and Serge Vl\u{a}du\c{t}.
\newblock A non-classical solution to a {H}essian equation from {C}artan
  isoparametric cubic.
\newblock {\em Adv. Math.}, 231(3-4):1589--1597, 2012.

\bibitem{NV3}
Nikolai Nadirashvili and Serge Vl\u{a}du\c{t}.
\newblock Nonclassical solutions of fully nonlinear elliptic equations.
\newblock {\em Geom. Funct. Anal.}, 17(4):1283--1296, 2007.

\bibitem{NV2}
Nikolai Nadirashvili and Serge Vl\u{a}du\c{t}.
\newblock Singular viscosity solutions to fully nonlinear elliptic equations.
\newblock {\em J. Math. Pures Appl. (9)}, 89(2):107--113, 2008.

\bibitem{NV1}
Nikolai Nadirashvili and Serge Vl\u{a}du\c{t}.
\newblock Octonions and singular solutions of {H}essian elliptic equations.
\newblock {\em Geom. Funct. Anal.}, 21(2):483--498, 2011.

\bibitem{NV4}
Nikolai Nadirashvili and Serge Vl\u{a}du\c{t}.
\newblock Singular solutions of {H}essian fully nonlinear elliptic equations.
\newblock {\em Adv. Math.}, 228(3):1718--1741, 2011.

\bibitem{nadirashvili2013}
Nikolai Nadirashvili and Serge Vl\u{a}du\c{t}.
\newblock Homogeneous solutions of fully nonlinear elliptic equations in four
  dimensions.
\newblock {\em Comm. Pure Appl. Math.}, 66(10):1653--1662, 2013.

\bibitem{NV5}
Nikolai Nadirashvili and Serge Vl\u{a}du\c{t}.
\newblock Singular solutions of {H}essian elliptic equations in five
  dimensions.
\newblock {\em J. Math. Pures Appl. (9)}, 100(6):769--784, 2013.

\bibitem{NV6}
Nikolai Nadirashvili and Serge Vl\u{a}du\c{t}.
\newblock Singular solutions to conformal {H}essian equations.
\newblock {\em Chin. Ann. Math. Ser. B}, 38(2):591--600, 2017.

\bibitem{NYuan}
Nikolai Nadirashvili and Yu~Yuan.
\newblock Homogeneous solutions to fully nonlinear elliptic equations.
\newblock {\em Proc. Amer. Math. Soc.}, 134(6):1647--1649, 2006.

\bibitem{nirenberg}
Louis Nirenberg.
\newblock On nonlinear elliptic partial differential equations and {H}\"{o}lder
  continuity.
\newblock {\em Comm. Pure Appl. Math.}, 6:103--156; addendum, 395, 1953.

\bibitem{safonov}
M.~V. Safonov.
\newblock Unimprovability of estimates of {H}\"{o}lder constants for solutions
  of linear elliptic equations with measurable coefficients.
\newblock {\em Mat. Sb. (N.S.)}, 132(174)(2):275--288, 1987.

\bibitem{lihe1}
Lihe Wang.
\newblock On the regularity theory of fully nonlinear parabolic equations. {I}.
\newblock {\em Comm. Pure Appl. Math.}, 45(1):27--76, 1992.

\bibitem{lihe2}
Lihe Wang.
\newblock On the regularity theory of fully nonlinear parabolic equations.
  {II}.
\newblock {\em Comm. Pure Appl. Math.}, 45(2):141--178, 1992.

\bibitem{WuZhang}
Jie Wu and Liqun Zhang.
\newblock Backward uniqueness for parabolic operators with variable
  coefficients in a half space.
\newblock {\em Commun. Contemp. Math.}, 18(1):1550011, 38, 2016.

\end{thebibliography}
\index{Bibliography@\emph{Bibliography}}%
\end{document}